%% file: arxiv_clean.tex
\documentclass[a4paper,USenglish,cleveref,autoref,thm-restate,nolineno]{socg-lipics-v2021}

\usepackage{amsfonts}
\usepackage{mathbbol}
\usepackage{enumerate}
\usepackage{booktabs}
\usepackage{url}
\usepackage{hyperref}
\usepackage{amssymb}
\usepackage{framed}
\usepackage{amsthm}
\usepackage{graphicx}
\usepackage{tikz}
\usepackage{tikz-cd}
\usepackage{amsmath}
\usepackage{color}
\usepackage{xcolor} 
\usepackage{kotex}
\usepackage{caption} 
\usepackage{subcaption} 
\usepackage{mathtools}
\usepackage{etoolbox}
\usepackage{extpfeil}
\usepackage{soul}
\usepackage[ruled,vlined, linesnumbered]{algorithm2e}
\usepackage[normalem]{ulem}
\usepackage[inline]{enumitem}
\usepackage{adjustbox}
\usepackage{wrapfig}

\hypersetup{
    colorlinks=true,
    urlcolor = {darkblue},
    linkcolor = {darkblue},
    citecolor = {darkblue},
    linktoc=all
}

\definecolor{darkblue}{rgb}{0.0, 0.0, 0.8}
\definecolor{darkred}{rgb}{0.8, 0.0, 0.0}
\definecolor{darkgreen}{rgb}{0.6, 0.15, 0.15}
\usetikzlibrary{positioning}
\usetikzlibrary{math}

\newcommand{\rk}{\mathrm{rk}}
\newcommand{\R}{\mathbb{R}}
\newcommand{\Z}{\mathbb{Z}}
\newcommand{\F}{k}
\newcommand{\db}{d_{\mathrm{B}}}

\newcommand{\vect}{\mathrm{vect}_k}
\newcommand{\Int}{\mathrm{Int}}
\newcommand{\N}{\mathbb{N}}
\newcommand{\Ical}{\mathcal{I}}
\newcommand{\Jcal}{\mathcal{J}}
\newcommand{\Lcal}{\mathcal{L}}
\newcommand{\Rcal}{\mathcal{R}}
\newcommand{\Dcal}{\mathcal{D}}
\newcommand{\Ecal}{\mathcal{E}}
\newcommand{\Kcal}{\mathcal{K}}
\newcommand{\dgm}{\mathrm{dgm}}

\newcommand{\barc}{\mathrm{barc}}
\newcommand{\id}{\mathrm{id}}
\newcommand{\rank}{\mathrm{rank}}

\newcommand{\dero}{d_{\mathrm{E}}}
\newcommand{\hatdero}{\hat{\dero}}
\newcommand{\dhat}{\hat{d}}

\newcommand{\dd}{d_{\ell}}
\newcommand{\dint}{d_{\mathrm{I}}}
\newcommand{\abs}[1]{\left\lvert{#1}\right\rvert}
\newcommand{\dw}[1]{d_{\mathrm{W},#1}}

\newcommand{\gpdpairs}{GPDs relative to sampled intervals}
\newcommand{\gpdpairsdiff}{GPDs relative to (possibly different) sampled intervals}

\newcommand{\codelink}{\href{https://github.com/L-ebesgue/sparse_GPDs}{sparse GPDs}}

\hideLIPIcs

\title{Sparsification of the Generalized Persistence Diagrams for Scalability through Gradient Descent} 

\author{Mathieu Carrière}{Centre Inria d'Université Côte d'Azur, Sophia Antipolis, France}{mathieu.carriere@inria.fr}{}{Partially supported by ANR grant “TopModel”, ANR-23-CE23-0014, and supported by the French government, through the 3IA Cote d’Azur Investments in the project managed by the National Research Agency (ANR) with the reference number ANR-23-IACL-0001.}
\author{Seunghyun Kim}{Department of Mathematical Sciences, KAIST, South Korea}{leo_k@kaist.ac.kr}{}{}
\author{Woojin Kim}{Department of Mathematical Sciences, KAIST, South Korea}{woojin.kim@kaist.ac.kr}{}{Partially supported by the National Research Foundation of Korea (NRF) grant funded by the Korea government(MSIT) (RS-2025-00515946).}

\authorrunning{M. Carrière, S. Kim and W. Kim}
\Copyright{M. Carrière, S. Kim and W. Kim}



\acknowledgements{This work stemmed from a conversation between M.C. and W.K. during the Dagstuhl Seminar on `Applied and Combinatorial Topology' (24092), held from February 26 to March 1, 2024. M.C. and W.K. thank the organizers of the Dagstuhl Seminar and appreciate the hospitality of Schloss Dagstuhl – Leibniz Center for Informatics. The authors are also grateful to the OPAL infrastructure from Université Côte d’Azur for providing resources and support.
}

\nolinenumbers

\begin{document}

\maketitle

\begin{abstract}
    The generalized persistence diagram (GPD) is a natural extension of the classical persistence barcode to the setting of multi-parameter persistence and beyond. The GPD is defined as an integer-valued function whose domain is the set of intervals in the indexing poset of a persistence module, and is known to be able to capture richer topological information than its single-parameter counterpart. However, computing the GPD 
    is computationally prohibitive due to the sheer size of the interval set. Restricting the GPD to a subset of intervals provides a way to manage this complexity, compromising discriminating power to some extent. However, identifying and computing an effective restriction of the domain that minimizes the loss of discriminating power remains an open challenge.

    In this work, we introduce a novel method for optimizing the domain of the GPD through gradient descent optimization. To achieve this, we introduce a loss function tailored to optimize the selection of intervals, balancing computational efficiency and discriminative accuracy. The design of the loss function is based on the known erosion stability property of the GPD. We showcase the efficiency of our sparsification method for dataset classification in supervised machine learning. Experimental results demonstrate that our sparsification method significantly reduces the time required for computing the GPDs associated to several datasets, while maintaining classification accuracies comparable to those achieved using full GPDs. Our method thus opens the way for the use of GPD-based methods to applications at an unprecedented scale.
\end{abstract}

\section{Introduction}
Persistent homology, a central tool in topological data analysis (TDA), enables the study of topological features in datasets through algebraic invariants. In the classical one-parameter setting, the persistence barcode (or equivalently, the persistence diagram) serves as a complete, discrete, and computationally tractable invariant of a persistence module. Persistent homology can be extended to multi-parameter persistent homology, which provides tools for capturing the topological features of datasets using multiple filtrations instead of just one. However, the transition to multi-parameter persistent homology introduces significant complexity into the algebraic structure of the associated persistence modules \cite{bauer2020cotorsion,botnan2022introduction,carlsson2009theory}.

Nevertheless, the generalized persistence diagram (GPD) extends the notion of persistence diagram from the one-parameter to the multi-parameter setting in a natural way \cite{kim2021generalized,patel2018generalized}. Although the GPD has been extensively studied in terms of stability, discriminating power, computation, and generalizations (see, e.g., \cite{asashiba2024interval,clause2022discriminating,dey2022computing,dey2024computing,kim2021generalized,kim2023persistence,kim2024extracting,kim2021bettis}), the computational complexity of GPDs remains a major obstacle \cite{kim2024superpoly}.
The primary challenge arises from the size of their domain: the domain of the complete GPD is either $\Int(\R^d)$, the set of all intervals in $\R^d$, or any appropriately chosen, finite subset of $\Int(\R^d)$---however, in this case, domains still tend to be enormous to avoid sacrificing the GPD discriminating power \cite{clause2022discriminating}.
Nevertheless, GPDs are flexible, in the sense that they are still well-defined on \emph{any finite} subdomain $\Ical\subset \Int(\R^d)$. Even if the subdomain has small size, as the GPD over $\Ical$ is simply defined as the M\"obius inversion of the generalized rank invariant (GRI) over $\Ical$ (see \Cref{def:GPD}). 
This allows to control the aforementioned complexity of GPD computation by picking a small, or \emph{sparse}, subdomain $\Ical$. Moreover, the topological information loss due to using such sparse subdomains 
can be mitigated by looking for \emph{relevant} intervals, i.e., intervals that are rich in topological content; indeed, even with a substantially small subdomain $\Ical\subset \Int(\R^d)$,
the GPD over $\Ical$ can still be finer than other traditional invariants of multi-parameter persistence modules, such as the rank invariant (RI) \cite{carlsson2009theory};
see \cite{clause2022discriminating} for details. 
However, how to design these subdomains in the ``best'' way so as to reduce computational cost while maintaining the discriminating power of the GPD is
a question that has not been much explored so far.

\subparagraph*{Sparsification of the GPD via gradient descent.} 
Motivated by the computational challenges of the GPD and the flexibility of this invariant upon selecting subsets of intervals as its domain, we propose a method for \emph{automatically sparsifying GPDs} computed from $\R^2$-persistence modules based on gradient descent. Namely, we consider the following scenario.

Suppose we aim at classifying instances of a given dataset based on their topological features, and we have already computed a set of corresponding persistence modules $\{M_i:\R^2\to \vect\}_{1 \leq i \leq t}$ from the dataset, where each persistence module corresponds to an individual data point. We consider the set  $\{\dgm_{M_i}^{\Ical}\}_{1\leq i\leq t}$ of GPDs of $M_i$ over a large set $\Ical$ of intervals in $\R^2$ (cf. \Cref{def:GPD}).
We refer to these as the \emph{full GPDs}. 

Let $n \gg 1$ be the cardinality of $\Ical$, and let $m$ be a sparsification parameter $m \in \mathbb{N}^*=\{1,2,\ldots\}$, which is typically significantly smaller than $n$. 
Let $\binom{\Int(\R^2)}{m}$ denote the set of $m$-subsets of $\Int(\R^2)$, i.e., $\{\Jcal\subset \Int(\R^2):\abs{\Jcal}=m\}$. In order to identify an $m$-subset of intervals in $\R^2$ over which the new GPDs are computed (and subsequently used to classify the persistence modules $\{M_i\}_{1 \leq i \leq t}$), we proceed as follows. \textbf{Firstly,} 
we identify a loss function \emph{defined on $\binom{\Int(\R^2)}{m}$}:
\begin{equation}\label{eq:ideal_loss}
\def\arraystretch{1.2}\begin{array}{rrcl}
    \Lcal_{\dd,m,\{M_i\}_{1\leq i \leq t}}: & \binom{\Int(\R^2)}{m} &\rightarrow& \R \\
    & \Jcal &\mapsto& \sum_{i=1}^{t} \dd(\dgm_{M_i}^\Jcal, \dgm_{M_i}^\mathcal{I}),
\end{array}
\end{equation}
where $\dd$ is an appropriate dissimilarity function. \textbf{Secondly,} we search for a minimizer of the loss function. 
The goal of this search is to identify a subset $\Jcal^\ast$ of $m$ intervals such that the GPDs of the $\{M_i\}_i$ over this sparse subset $\{\dgm_{M_i}^{\Jcal^\ast}\}_{1\leq i\leq t}$,
the \emph{sparse GPDs}, best approximate their corresponding full counterparts overall. 
One natural way of searching for a (local) minimizer is through \textbf{gradient descent}, starting from a randomly chosen $m$-subset $\Jcal_{\text{init}}$ of $\Ical$. 
To achieve this, the following requirements either must be met or highly desirable:
\begin{enumerate}[label=(\Roman*)]
\item (Distance)
A suitable distance or dissimilarity function $\dd$, utilized in constructing the loss function above, ideally satisfying certain stability guarantees w.r.t. the \emph{interleaving distance} between persistence modules \cite{chazal2009proximity,lesnick2015theory},
\label{goal:good_distance}
\item (Vectorization) 
A certain representation of the loss function $\Lcal_{\dd,m,\{M_i\}_{1\leq i \leq t}}$ as a map defined on some subset $\Dcal$ of Euclidean space,\label{goal:vectorization}
\item (Convexity) 
Convexity of the subset $\Dcal$,
\label{goal:convexity}
\item (Loss regularity)  Lipschitz stability and differentiability of 
$\Lcal_{\dd,m,\{M_i\}_{1\leq i \leq t}}$, and \label{goal:continuity}
\item (Feasibility)  Computational feasibility of $\Lcal_{\dd,m,\{M_i\}_{1\leq i \leq t}}$ for practical implementation.
\label{goal:feasibility}
\end{enumerate}

Our contributions can be listed according to Items \ref{goal:good_distance}-\ref{goal:feasibility} mentioned above.

\subparagraph*{Summary of contributions.}

\begin{itemize}[leftmargin=1pt]
\item A natural choice for $\dd$ in Item \ref{goal:good_distance} would be to use the \emph{erosion distance} $\dero$ \cite{clause2022discriminating,patel2018generalized}, a standard metric between GPDs,\footnote{
Note that $\dero$ is also referred to as a metric between generalized rank invariants (GRIs), e.g., in \cite{clause2022discriminating,kim2021spatiotemporal}. Since the GPD and the GRI 
determine one another (cf. Remark \ref{rem:monotinocity} \ref{item:finite-then-exists} and \ref{item:exists-then-determine-each-other}), $\dero$ can thus also be viewed as a metric between GPDs, as in \cite{patel2018generalized,xian2022capturing}.} 
that is known to be stable under perturbations of input persistence modules w.r.t. the interleaving distance. 
However, the use of the erosion distance requires GPDs to be defined on the \emph{same} set of intervals (Definition \ref{def:erosion}) that is \emph{closed under thickening}, which implies the domain must be \emph{infinite}. All these make it difficult to directly utilize $\dero$.
Hence, \textbf{we introduce the sparse erosion distance $\hatdero$ between GPDs relative to (possibly) \emph{different} interval sets,
as an adaptation of $\dero$ (\Cref{def:hat_erosion},  \Cref{eros=intleav}~\ref{item:hat_generalizes_dE}, and \Cref{cor:stability})}.\footnote{Another possibility would be to use bottleneck and Wasserstein distances, however we show in \Cref{rmk:dB_and_dW_discontinuity} in the appendix that they fail to ensure continuity of the loss function.}

\item Regarding \Cref{goal:feasibility}, 
we restrict our focus to intervals in $\R^2$ with a small number of minimal and maximal points (\cref{rmk:computational_difficulty_when_pq_gets_larger}).
Then, if the sparse erosion distance $\hatdero$ is computed between GPDs of the \emph{same} persistence module $M$, 
\textbf{we prove that $\hatdero$ depends solely on the domains $\Ical, \Jcal$, and \emph{not} 
on the input persistence module, i.e. $\hatdero((\dgm_M,\Ical),(\dgm_M,\Jcal)) =\dhat(\Ical, \Jcal)$ (\textbf{\Cref{eros=intleav}~\ref{item:eros-intleav}})}. 
This fact significantly enhances the tractability of gradient descent as it allows to avoid recomputing the GPDs at every iteration.
Moreover, \textbf{we derive a closed-form formula for the computation of $\dhat(\Ical, \Jcal)$ (\Cref{thm_distance_btw_collections_intvs} \ref{item:between_p,q_intervals}).}

\begin{figure}[h]
    \centering
    \includegraphics[width=0.3\linewidth]{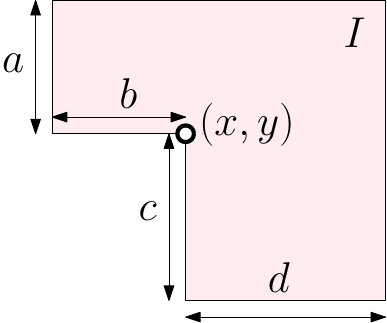}
    \caption{\label{fig:embedding} Any $(2,1)$-interval $I$ of $\R^2$, as depicted above, is represented by $\mathbf{v}_{I}=(x,y,a,b,c,d)\in\R^{6}$. Any $(1,1)$-interval can also be represented by $(x,y,a,b,c,d)\in\R^{6}$ with $b=c=0$.}
\end{figure}

\begin{figure}[h]
    \centering
    \includegraphics[width=\textwidth]{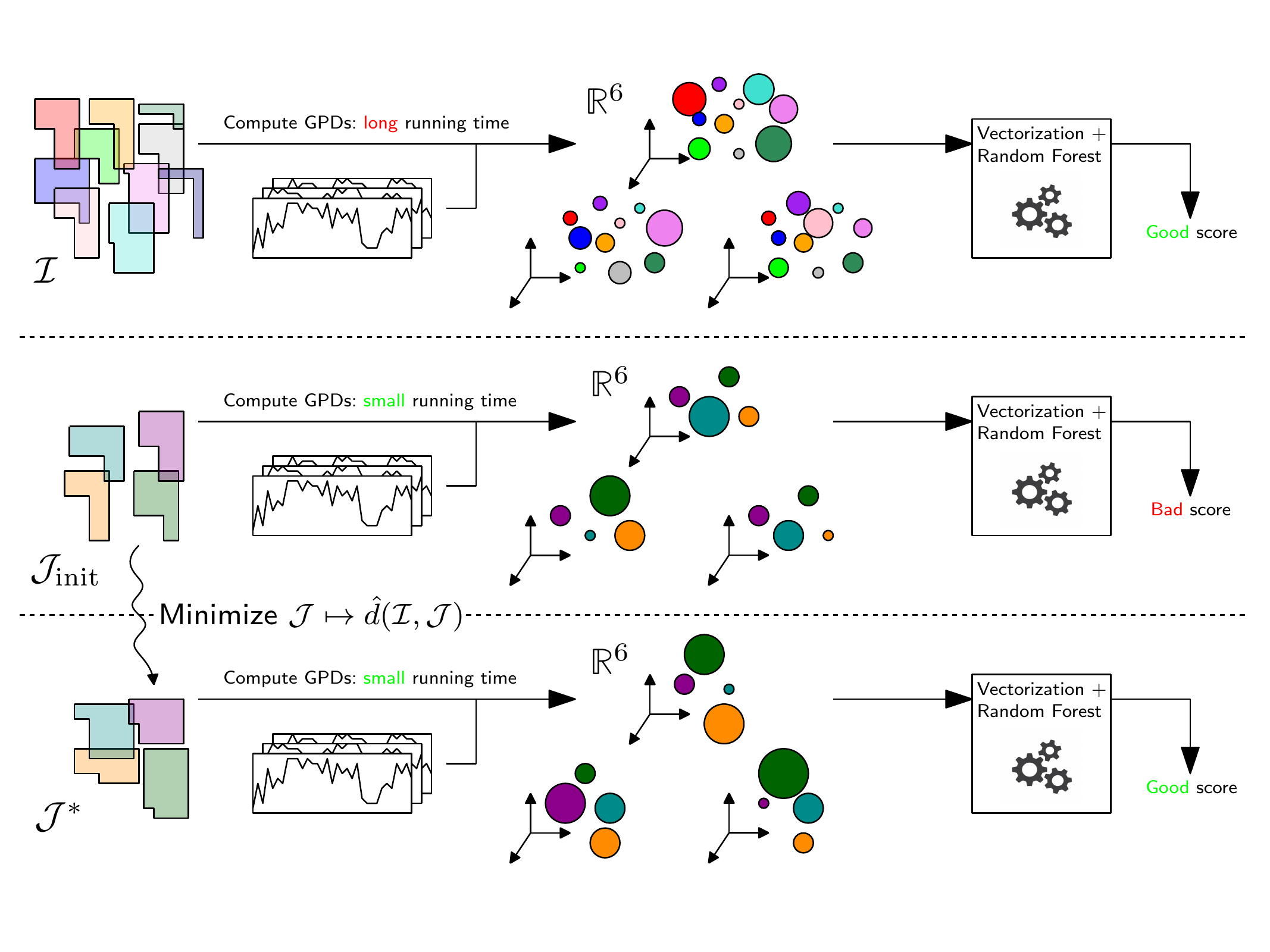}
    \caption{Our pipeline for sparsifying GPDs in the context of time series classification. The sizes of the GPD points in $\mathbb{R}^6$ are proportional to their multiplicities. 
    }
    \label{fig:pipeline}
\end{figure} 

\item 
For achieving \Cref{goal:vectorization,goal:convexity,goal:continuity}, we only consider intervals with at most two minimal points and one maximal point.
This ensures the existence of natural embeddings of these intervals into Euclidean space $\R^6$, which can be obtained by stacking up the coordinates of the interval middle points $(x,y)$ together with the lengths $a,b,c,d\geq 0$ needed to define the interval lower boundaries (\Cref{fig:embedding}).
The main advantages of this vectorization method (w.r.t. the other natural ones) are that (1) it allows for \textbf{a simple formulation of the loss function (\Cref{thm_distance_btw_collections_intvs} \ref{item:between_2,1_intervals})}, and (2) its variables are all independent from each other, which makes the implementation of gradient descent easier (Remarks \ref{rmk:computational_difficulty_when_pq_gets_larger} and \ref{rmk:convex_vectorization_for_larger_pq}).\footnote{Alternatively, we can consider intervals with one minimal points and at most two maximal points. Either choice ensures that our restricted GPDs are not a weaker invariant than the rank invariant or the signed barcode \cite{botnan2021signed}; see \cite[Section 4]{clause2022discriminating}.}
Using this vectorization method, \textbf{we also prove that $\Dcal$ not only forms a convex subset of Euclidean space $\R^{6m}$ (\textbf{\Cref{prop:convexity}}), but also ensures the Lipschitz stability and almost everywhere differentiability of the loss function $\Lcal_{\dd,m,\{M_i\}_{1\leq i \leq t}}$ (\Cref{thm:de-vectorization_stability} and \Cref{prop:differentiable})}. 
In fact, we prove that the loss function can actually be realized as
\begin{equation}\label{eq:practical_loss}
\def\arraystretch{1.2}\begin{array}{rrcl}
    \Lcal_{\hatdero,m}: & \Dcal(\subset \R^{6m}) &\rightarrow& \R \\
    & \mathbb{v}_\Jcal &\mapsto& t\cdot\dhat(\Ical,\Jcal),
\end{array}
\end{equation}
where $\mathbb{v}_{\Jcal}$ is our $6m$-dimensional embedding of $\Jcal$. 
Since $t$ (the size of the dataset) is a constant, the (local) minimizers of $\Lcal_{\hatdero,m}$ coincide with those of the map $t^{-1}\cdot\Lcal_{\hatdero,m}: \mathbb{v}_\Jcal \mapsto \dhat(\Ical, \Jcal)$. 
Hence, searching for a (local) minimizer of $\Lcal_{\hatdero,m}$ is essentially searching for the (locally) best $m$ intervals that represent the domain $\Ical$ of the full GPDs.
 
\item  
Finally, regarding \Cref{goal:feasibility} again, and in order to showcase the efficiency of our proposed sparsification method, \textbf{we provide numerical experiments on topology-based time series 
classification with supervised machine learning (\Cref{sec:expes} and \Cref{fig:pipeline})}, in
which we show that the sparse GPDs $\{\dgm_{M_i}^{\Jcal^\ast}\}_{1\leq i\leq t}$ can be computed in much faster time than the full GPDs, while maintaining similar or better classification performances from random forests models. Our code is fully available at \codelink.
\end{itemize}

\subparagraph*{Comparison with other works.}
While there exist many existing works that utilize multi-parameter persistent homology to enhance the performance of machine learning models \cite{Carriere2020b, Corbet2019, Loiseaux2023b, loiseaux2024stable, Mukherjee2024, Scoccola2024, Vipond2020, xin2023gril}, our work takes a different approach as it focuses on methods to mitigate the computational overhead associated with multi-parameter persistence descriptors.

The works most closely related to ours are~\cite{loiseaux2024stable} and \cite{xin2023gril}, which use the RI or GRI for multi-parameter persistence modules in a machine learning context.
Firstly, \cite{loiseaux2024stable} is based on an equivalent representation of the GPDs (computed from rectangle intervals\footnote{This types of GPDs are often called signed barcodes.}) as \emph{signed measures}~\cite{botnan2021signed}, which allows to compare GPDs with optimal transport distances, as well as to deploy known vectorization techniques intended for general measures. Secondly, the approach proposed in \cite{xin2023gril} involves vectorizing the GRIs 
of 2-parameter persistence modules by evaluating them on intervals with specific shapes called \emph{worms}.
Our goal is different: we rather aim at vectorizing and sparsifying the \emph{domains} of the GPDs in order to achieve good performance scores. In fact, the sparsification process that we propose can actually be used complementarily to both \cite{loiseaux2024stable} and \cite{xin2023gril} by first sparsifying the set of rectangles or worms before applying their vectorization methods.
Note that differentiability properties of both of these approaches \emph{w.r.t. the multi-parameter filtrations} were recently established~\cite{Mukherjee2024,Scoccola2024}; in contrast, our work deals with the differentiability of a GPD-based loss function \emph{w.r.t. the interval domains} (while keeping the multi-parameter filtrations fixed).

\subparagraph*{Organization.} \Cref{sec:preliminaries} reviews  basic properties of the GPD and GRI. \Cref{sec_dhat} introduces the \emph{sparse} erosion distance, clarify its relation to the erosion distance given in \cite{clause2022discriminating}, and presents its closed-form formula, which is specialized and useful in our setting. \Cref{sec:vectorization} establishes the Lipschitz stability and differentiability of our loss function. \Cref{sec:expes} presents our numerical experiments. Finally, \Cref{sec:conclusion}  discusses future research directions.

\section{Preliminaries}\label{sec:preliminaries}
In this article, $P=(P,\leq)$ stands for a poset, regarded as the category whose objects are the elements of $P$, and for any pair $p,q\in P$, there exists a unique morphism $p\to q$ if and only if $p\leq q$.
All vector spaces in this article are over a fixed field $\F$. 
Let $\vect$ denote the category of finite-dimensional vector spaces and linear maps over $\F$. A functor $P \to \vect$ will be referred to as a \textbf{($P$-)persistence module}. 
The \textbf{direct sum} of any two $P$-persistence modules is defined pointwise. 
Any $M:P\rightarrow \vect$ is \textbf{trivial} if $M(x)=0$ for all $x\in P$. If a nontrivial $M$ is not isomorphic to a direct sum of any two nontrivial persistence modules, $M$ is \textbf{indecomposable}. Every persistence module is decomposed into a direct sum of indecomposable modules, uniquely determined up to isomorphism \cite{azumaya1950corrections,botnan2020decomposition}.

An \textbf{interval} $I$ of $P$ is a subset $I\subseteq P$ such that: 	
\begin{enumerate*}[label=(\roman*)]
    \item $I$ is nonempty.    
    \item If $p,q\in I$ and $p\leq r\leq q$, then $r\in I$. \label{item:convexity}
    \item $I$ is \textbf{connected}, i.e. for any $p,q\in I$, there is a sequence $p=p_0,
		p_1,\cdots,p_\ell=q$ of elements of $I$ with $p_i$ and $p_{i+1}$ comparable for $0\leq i\leq \ell-1$.\label{item:interval3}
\end{enumerate*}
By $\Int(P)$ we denote the set of all intervals of $P$.

Given any $I\in \Int(P)$, the \textbf{interval module} $\F_I$ is the $P$-persistence module, with
\begin{equation*}\label{eq:interval module} 
    (\F_I)(p) := \begin{cases} \F & \mathrm{if \ } p\in I\\
    0 & \mathrm{otherwise.} \end{cases},
    \hspace{10mm}
    \F_I(p\leq q) := \begin{cases} \id_\F & \mathrm{if \ } p\leq q\in I\\
    0 & \mathrm{otherwise.}\end{cases}
\end{equation*}

Every interval module is indecomposable \cite[Proposition 2.2]{botnan2018algebraic}. \label{nom:barcode} 
A $P$-persistence module $M$ is \textbf{interval-decomposable} if it is isomorphic to a direct sum $\bigoplus_{j\in J}\F_{I_j}$ of interval modules. In this case, the \textbf{barcode} of $M$ is defined as the multiset $\barc(M):=\{\F_{I_j}:j\in J\}$.

For $p\in P$, let $p^{\uparrow}$ denote the set of points $q\in P$ such that $p\leq q$. Clearly, $p^\uparrow$ belongs to $\Int(P)$.
A $P$-persistence module is \textbf{finitely presentable} if it is isomorphic to the cokernel of a morphism 
$\bigoplus_{a \in A} \F_{a^\uparrow} \to \bigoplus_{b \in B}  \F_{b^\uparrow},$ where $A$ and $B$ are finite multisets of elements of $P$.

When $P$ is a connected poset, the \textbf{(generalized) rank} of $M$, denoted by  $\rank(M)$, is defined as the rank of the canonical linear map from the limit of $M$ to the colimit of $M$, which is a nonnegative integer \cite{kim2021generalized}. This isomorphism invariant of $P$-persistence modules, which takes a single integer value, can be refined into an integer-valued function as follows.
Let $\Ical$ be any nonempty subset of $\Int(P)$. The \textbf{generalized rank invariant (GRI) of $M$ over $\Ical$} is the map $\rk_M^\Ical:\Ical \to \Z_{\geq 0}$ given by
$I\mapsto \rank(M\vert_I)$, where $M\vert_I$ is the restriction of $M$ to $I$ \cite{kim2021generalized}. 
When $\Ical=\Int(P)$, we denote $\rk_M^\Ical$ simply as $\rk_M$.

The generalized persistence diagram (GPD) of $M$ over $\Ical$ captures the changes of the GRI values when $I\in \Ical$ varies. Its formal definition follows.
\begin{definition}[{\cite{clause2022discriminating}}]\label{def:GPD}
    The \textbf{generalized persistence diagram (GPD) of $M$ over $\Ical$} is defined as the function $\dgm_M^\Ical:\Ical\rightarrow \Z$ that satisfies:\footnote{The condition in Equation \eqref{eq:rk_in_terms_of_dgm} is a generalization of \emph{the fundamental lemma of persistent homology} \cite{edelsbrunner2008computational}.} 
    \begin{equation}\label{eq:rk_in_terms_of_dgm}
        \mbox{for all } I\in \Ical,\hspace{3mm} \rk_M^\Ical(I)=\sum_{\substack{J\in \Ical\\ J\supseteq I}} \dgm_M^\Ical(J).
    \end{equation}
\end{definition}

\begin{remark}[{\cite[Sections 2 and 3]{clause2022discriminating}} and {\cite{botnan2021signed}}]\label{rem:monotinocity} 
\begin{enumerate}[label=(\roman*)]
    \item If $\Ical$ is finite, then $\dgm_M^\Ical$ exists.\label{item:finite-then-exists}
    \item If $\dgm_M^\Ical$ exists, then it is unique. \label{item:exists-then-unique} 
    \item If $\dgm_M^\Ical$ exists, then $\dgm_M^\Ical$ and $\rk_M^\Ical$ determine one another.\label{item:exists-then-determine-each-other}
    \item (Monotonicity) $\rk_M^\Ical(I)\leq \rk_M^\Ical(J)$ for any pair $I\supseteq J$ in $\Ical$. \label{item:monotonicity}
    \item (The GPD generalizes the barcode) Let $\Ical=\Int(P)$. If $M$ is interval decomposable, then $\dgm_M^\Ical$ exists. In this case, for any $I\in \Int(P)$,  $\dgm_M^\Ical(I)$ coincides the multiplicity of $I$ in $\barc(M)$. Also, $\dgm_M^\Ical$ often exists even when $M$ is not interval decomposable. \label{item:dgm-generalizes-barc}
    \item If $M$ is a \emph{finitely presentable} $\R^d$-persistence module, then the GPD over $\Int(\R^d)$ exists \cite[Theorem C(iii)]{clause2022discriminating}. \label{rem:fp-setting-GPD-exists}
\end{enumerate} 
\end{remark}

In the rest of the article, every $\R^d$-persistence module $M$ is assumed to be finitely presentable, thus its GPD over $\Int(\R^d)$ exists, and is denoted by $\dgm_M$.

\section{Sparse erosion distance between \gpdpairs} \label{sec_dhat}
We adapt the notion of erosion distance to define a distance between \emph{\gpdpairsdiff}. When comparing the same GPD \emph{with two different sampled intervals}, this distance simplifies to a distance between sampled intervals. This is relevant in our case, as we compare the full and sparse GPDs of the \emph{same} persistence module.

\subsection{Sparse erosion distance}
In this section, we review the definition of the erosion distance and adapt it to define the sparse erosion distance between \emph{\gpdpairs}.

For $\epsilon \in \R$, the vector $\epsilon(1, \ldots, 1) \in \R^d$ will be simply denoted by $\epsilon$ whenever there is no risk of confusion.
For $I\in \Int(\R^d)$ and $\epsilon \in \R_{\geq0}$, we consider the \emph{$\epsilon$-thickening} $I^{\epsilon} := \bigcup_{p \in I} B_{\epsilon}^\square(p)$ of $I$, where $B_{\epsilon}^\square(p)$ stands for the closed $\epsilon$-ball around $p$ w.r.t. the supremum distance, i.e. $B_{\epsilon}^\square(p)=\{q \in \R^d \mid p - \epsilon \leq q \leq p + \epsilon\}$. See Figure~\ref{fig:int-eps}.
A subset $\Ical\subset \Int(\R^d)$ is said to be \textbf{closed under thickening} if for all $I\in \Ical$ and for all $\epsilon \in \R_{\geq0}$, the interval $I^\epsilon $ belongs to $\Ical$.

\begin{figure}[h]
    \centering
    \includegraphics[width=.25\textwidth]{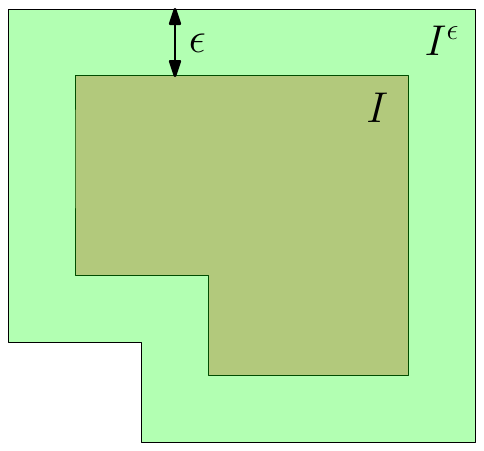}
    \caption{}
    \label{fig:int-eps}
\end{figure}
\begin{definition}[{\cite[Definition 5.2]{clause2022discriminating}}] \label{def:erosion} 
    Let $M$ and $N$ be $\R^d$-persistence modules. Let $\Ical$ be any subset of $\Int(\R^d)$ \emph{that is closed under thickening}. 
    The \textbf{erosion distance} between $\dgm_M^\Ical$ and $\dgm_N^{\Ical}$ (and equivalently between $\rk_M^\Ical$ and $\rk_N^{\Ical}$ by Remark \ref{rem:monotinocity} \ref{item:exists-then-unique}) is 
    \begin{equation*}
        \dero(\dgm_M^\Ical,\dgm_N^\Ical) :=\inf(\epsilon > 0:\mbox{for all } I\in \Ical,\ \rk_N(I^{\epsilon}) \leq \rk_M(I) \mbox{ and } \rk_M(I^{\epsilon}) \leq \rk_N(I)).
    \end{equation*}
\end{definition}
A \textbf{correspondence} between nonempty sets $A$ and $B$ is a subset  $\mathcal{R}\subset A \times B$ satisfying the following: (1) for each $a\in A$, there exists $b\in B$ such that $(a,b)\in \mathcal{R}$, and (2) for each $b\in B$, there exists $a\in A$ such that $(a,b)\in \mathcal{R}$.
For $\epsilon\in \R_{\geq0}$,  an \textbf{$\epsilon$-correspondence} $\mathcal{R}$ between nonempty $\Ical,\Jcal \subset \Int(\R^d)$ is a correspondence $\mathcal{R} \subset \mathcal{I} \times \mathcal{J}$ such that for all $(I, J) \in \mathcal{R}$, $J \subset I^{\epsilon} \; \mbox{ and } \; I \subset J^{\epsilon}.$
Blending the ideas of the Hausdorff and erosion distances, we obtain our new distance between \emph{\gpdpairs}:
\begin{definition}[Sparse erosion distance between \gpdpairs] 
\label{def:hat_erosion}
    For any $M,N:\R^d\rightarrow \vect$ and any nonempty $\Ical,\Jcal\subset \Int(\R^d)$, the \textbf{sparse erosion distance} between pairs $(\dgm_M,\Ical)$ and $(\dgm_N,\Jcal)$ is
    \begin{multline*}
        \hatdero((\dgm_M,\Ical),(\dgm_N,\Jcal))
        := \inf(\epsilon > 0: \mbox{ there exists an } \epsilon \textup{-correspondence}\; \mathcal{R} \subset \mathcal{I} \times \mathcal{J} \mbox{ s.t.} \\
        \forall (I, J) \in \mathcal{R},\  \forall \delta \in \R_{\geq 0},\; \rk_N(J^{\epsilon+\delta}) \leq \rk_M(I^\delta) \mbox{ and } \rk_M(I^{\epsilon+\delta}) \leq \rk_N(J^\delta)).
    \end{multline*}
\end{definition}
We remark that $\hatdero((\dgm_M,\Ical),(\dgm_N,\Jcal))$ captures not only (1) the algebraic difference between [$M$ with respect to $\Ical$] and [$N$ with respect to $\Jcal$], but also (2) the geometric difference between the domains $\Ical$ and $\Jcal$.
To see (1), consider, for example, the $\R$-persistence modules $M = \F_{[0,1]}$ and $N = \F_{[0,2]}$, and let both $\Ical$ and $\Jcal$ be the singleton set $\{[3,4]\}$. Then, one can see that $\hatdero((\dgm_M,\Ical),(\dgm_N,\Jcal)) = 0$ from the fact that $\rk_M^{\Ical} = \rk_N^\Jcal = 0$ and the monotonicity of $\rk_M$ and $\rk_N$.
To see (2), let $M$ and $N$ be any isomorphic $\R$-persistence modules, and set $\Ical := \{[0,\frac{1}{2}]\}$ and $\Jcal := \{[0,1]\}$. Then, we obtain $\hatdero((\dgm_M,\Ical),(\dgm_N,\Jcal)) \geq 1/2$, solely due to the difference between $\Ical$ and $\Jcal$.

\begin{proposition} \label{prop:pseudo}
    $\hatdero$ is an extended pseudometric. (See \Cref{A:pseudo} for the proof.)
\end{proposition}
For any nonempty $\Ical,\Jcal \subset \Int(\R^d)$, let
\begin{equation} \label{def_dhat} 
    \dhat(\mathcal{I}, \mathcal{J}) := \inf(\epsilon > 0 :\; \textup{there exists an} \; \epsilon \textup{-correspondence between }  \mathcal{I} \textup{ and } \mathcal{J}).
\end{equation}
We clarify the relationship among $\dero$, $\hatdero$, and $\dhat$:
\begin{proposition} \label{eros=intleav}
    Let $M,N:\R^d\rightarrow \vect$ and let $\Ical,\Jcal\subset \Int(\R^d)$ be nonempty. We have:
    \begin{enumerate}[label=(\roman*)]
        \item $\dhat(\mathcal{I}, \mathcal{J}) \leq \hatdero((\dgm_M,\Ical),(\dgm_N,\Jcal))$. \label{item:dEhat_is_less_than_dhat}
        \item \label{item:hat_generalizes_dE} If $\Ical=\Jcal$ are closed under thickening, then 
        \begin{equation}\label{eq:hat_generalizes_dE}
            \hatdero((\dgm_M,\Ical),(\dgm_N,\Jcal)) \leq \dero(\dgm^\mathcal{I}_M, \dgm^\mathcal{I}_N). 
        \end{equation}
        \item \label{item:eros-intleav} If $M \cong N$ or $\dgm_M=\dgm_N$, then 
        \begin{equation}\label{eq:eros-intleav}
            \hatdero((\dgm_M,\Ical),(\dgm_N,\Jcal)) = \dhat(\mathcal{I}, \mathcal{J}).
        \end{equation}
    \end{enumerate}(See \Cref{A:eros=intleav} for the proof.)
\end{proposition}
We remark that the inequality given in \Cref{item:dEhat_is_less_than_dhat} can be strict. For instance, let $d=1$, $\Ical = \Jcal=\Int(\R^d)$, $M = 0$,  and $N=k_{[0,1)}$. Then, $0=\dhat(\Ical,\Jcal)<\hatdero((\dgm_M,\Ical),(\dgm_N,\Jcal))$.

By \Cref{eros=intleav} \ref{item:hat_generalizes_dE} and the prior stability result of $\dero$ \cite[Theorem H]{clause2022discriminating}, we have:
\begin{corollary}\label{cor:stability}
    For any $M,N:\R^d\rightarrow \vect$ and for any $\Ical\subset \Int(\R^d)$, we have 
    \[\hatdero((\dgm_M,\Ical),(\dgm_N,\Ical))\leq \dint(M,N),\] 
    where the right-hand side is the interleaving distance between $M$ and $N$.
\end{corollary}

\subsection{Closed-form formula for the sparse erosion distance} 
In this section, we find a closed-form formula for the right-hand side of Equation \eqref{eq:eros-intleav}, when each interval in $\Ical$ and $\Jcal$ has only finitely many minimal and maximal points (Theorem \ref{thm_distance_btw_collections_intvs}). This result is essential when studying our loss function in \Cref{sec:vectorization} and optimizing the domains of the GPDs in \Cref{sec:expes}.

For $p,q \in \N^\ast$, an interval $I \in \Int(\R^2)$ is a $(p,q)$-interval if $I$ has exactly $p$ minimal points and exactly $q$ maximal points. For any interval $I \in \Int(\R^2)$, let $\min (I)$ (resp. $\max (I)$) be the set of all minimal (resp. maximal) elements of $I$.   For $(x, y) \in \R^2$, define
\begin{align} \label{eq:delta}
    \delta(x, y) &:= \left\{ \begin{array}{lcl} 1 & \mbox{if} & x \leq y \\ 0 & \mbox{if} & x > y. \end{array}\right.
\end{align}

\begin{lemma} \label{lem_epsrs} 
    For $p_r,q_r,p'_s,q'_s\in \N^\ast$, let $I_r,J_s\in \Int(\R^2)$ be $(p_r, q_r)$- and $(p'_s, q'_s)$-intervals, respectively. Then, for any $\epsilon \in \R_{\geq0}$, $J_s \subset I^{\epsilon}_r$ and $I_r \subset J^{\epsilon}_s$ if and only if \[
    \epsilon \geq \max\left(\max_k a_k, \max_l b_l, \max_i a'_i, \max_j b'_j\right), \ \  \mbox{where} \]
    \begin{itemize}
        \item  $i,j,k,l$ range from $1$ to $p_r,q_r,p'_s,q'_s$ respectively; $a_k$ and $b_l$ are defined as (\ref{eq_a}) and (\ref{eq_b}) respectively; 
        $a_i'$ and $b_j'$ are obtained by interchanging the roles of $i$ and $k$ in (\ref{eq_a}), and those of $j$ and $l$ in (\ref{eq_b}), respectively;
        \item 
         $ \min(I_r) =: \{(x^r_i, y^r_i) : 1 \leq i \leq p_r\},\quad
        \min(J_s) =: \{(x^s_k, y^s_k) : 1 \leq k \leq p'_s\},\\
        \max(I_r) =: \{(X^r_j, Y^r_j) : 1 \leq j \leq q_r\},\quad 
        \max(J_s) =: \{(X^s_l, Y^s_l) : 1 \leq l \leq q'_s\}.$
    \end{itemize}
    \begin{align}
        & \min_i \biggl( \max \Bigl( (1-\delta(x^r_i, x^s_k))|x^s_k - x^r_i|, (1-\delta(y^r_i, y^s_k))|y^s_k - y^r_i| \Bigl) \biggl) \label{eq_a} \\
        & \min_j \biggl( \max \Bigl( \delta(X^r_j, X^s_l)|X^s_l - X^r_j|, \delta(Y^r_j, Y^s_l)|Y^s_l - Y^r_j| \Bigl) \biggl) \label{eq_b} 
    \end{align}
\end{lemma}    
\emph{(See \Cref{A:lem_epsrs} for the proof.)}
Let $\Ical := \{I_r\}^n_{r=1}$ and $\Jcal := \{J_s\}^m_{s=1}$  be sets of intervals of $\R^2$ with only finitely many minimal and maximal points. Let $\Ecal$ the ($n\times m$)-matrix $(\epsilon_{rs})$ where $\epsilon_{rs}$ is the RHS of the inequality given in \Cref{lem_epsrs}, i.e.
\begin{equation} \label{eq:epsilons}
    \epsilon_{rs} = \max(\max_k a_k, \max_l b_l, \max_i a'_i, \max_j b'_j)
\end{equation}
and $i,j,k,l, a_k, b_l, a'_i, b'_j$ as defined in \Cref{lem_epsrs}. Next, we show that $\dhat(\Ical,\Jcal)$, as defined in \Cref{def_dhat}, equals the largest of the smallest elements across all rows and columns of $\Ecal$.
In particular, the following variables and functions are useful for describing the closed form formula for $\dhat(\Ical,\Jcal)$, when each of $\Ical$ and $\Jcal$ consists solely of $(1,1)$- or $(2,1)$-intervals.

When $I_r$ and $J_s$ are $(2, 1)$-intervals, as depicted in Figure \ref{fig:two (2, 1)-intervals}, let
\begin{align*}
    x_1 &:= x^r_2, & y_1 &:= y^r_1, & a &:= Y^r_1 - y^r_1, & b &:= x^r_2 - x^r_1,\\
    c &:= y^r_1 - y^r_2, & d &:= X^r_1 - x^r_2, & x_2 &:= x^s_2, & y_2 &:= y^s_1,\\
    e &:= Y^s_1 - y^s_1, & f &:= x^s_2 - x^s_1, & g &:= y^s_1 - y^s_2, & h &:= X^s_1 - x^s_2.
\end{align*}

\begin{figure}[h]
    \centering
    \includegraphics[width=0.5\linewidth]{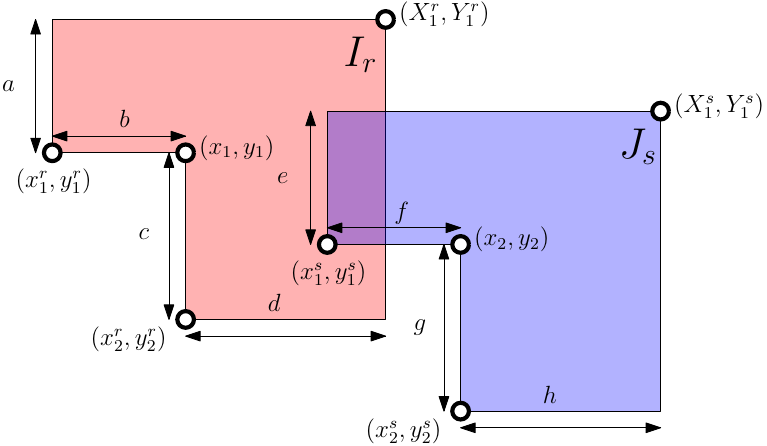}
    \caption{The parametrization of $I_r$ and $J_s$}
    \label{fig:two (2, 1)-intervals}
\end{figure}
When $I_r$ (resp. $J_s$) is a $(1, 1)$-interval, set $x^r_1 = x^r_2$, $y^r_1 = y^r_2$ and thus $b = c = 0$ (resp. $x^s_1 = x^s_2$, $y^s_1 = y^s_2$ and thus $f = g = 0$). Also, let
\begin{align*}
    F(w_1, w_2, w_3, w_4) &:= \max \bigl( \delta(w_1, w_2)|w_2 - w_1|, \delta(w_3, w_4)|w_4 - w_3| \bigl),\\
    G(m_1, m_2, m_3, m_4, m_5) &:= \min(F(m_1, m_2, m_3, m_4), m_5),\\
    H(o_1, o_2, o_3, o_4) &:= \max \bigl( \min(o_1, o_2), \min(o_3, o_4) \bigl),
\end{align*}
where all input variables of $F, G$ and $H$ are real numbers.
\begin{theorem} \label{thm_distance_btw_collections_intvs}
    Let $\Ical := \{I_r\}^n_{r=1}$ and $\Jcal := \{J_s\}^m_{s=1}$ be sets of intervals of $\R^2$ with only finitely many minimal and maximal points.
    \begin{enumerate}[label=(\roman*)]
    \item  We have $\dhat(\mathcal{I}, \mathcal{J}) = \max \Bigl(\max_r(\min_s \epsilon_{rs}), \max_s(\min_r \epsilon_{rs}) \Bigl)$ where $\epsilon_{rs}$ is as defined in \Cref{eq:epsilons} and $I_r$ (resp. $J_s$) is a $(p_r,q_r)$-interval (resp. $(p'_s,q'_s)$-interval). \label{item:between_p,q_intervals}
    \item If $\Ical$ and $\Jcal$ consist solely of $(1,1)$- or $(2,1)$-intervals, then $\epsilon_{rs}$ is the maximum of (\ref{eq_(2, 1)_a})-(\ref{eq_(2, 1)_b'})\label{item:between_2,1_intervals}.
    \end{enumerate} 
\end{theorem}

\begin{align}
    &H \Bigl( F(x_2 - f, x_1 - b, y_2, y_1), F(x_2-f, x_1, y_2, y_1-c), \nonumber \\
    &\ \ \ \ \ \ \ \ \ \ \ \ \ \ \ \ \ \ \ \ F(x_2, x_1-b, y_2-g, y_1), F(x_2, x_1, y_2-g, y_1-c)\Bigl) \label{eq_(2, 1)_a} \\
    &F(x_1+d, x_2+h, y_1+a, y_2+e) \label{eq_(2, 1)_b} \\
    &H \Bigl( F(x_1 - b, x_2 - f, y_1, y_2), F(x_1-b, x_2, y_1, y_2-g), \nonumber \\
    &\ \ \ \ \ \ \ \ \ \ \ \ \ \ \ \ \ \ \ \ F(x_1, x_2-f, y_1-c, y_2), F(x_1, x_2, y_1-c, y_2-g)\Bigl) \label{eq_(2, 1)_a'}\\
    &F(x_2+h, x_1+d, y_2+e, y_1+a) \label{eq_(2, 1)_b'}
\end{align}
\emph{(See Appendix \ref{A:pf_(2, 1)} for the proof.)}

\begin{remark} \label{rmk:computational_difficulty_when_pq_gets_larger} 
    In \Cref{thm_distance_btw_collections_intvs}~\ref{item:between_p,q_intervals}, as $p_r$, $q_r$, $p'_s$, and $q'_s$ increase, the complexity of computing $\epsilon_{rs}$ increases. This is because, the ranges over which the maxima on the RHS of \Cref{eq:epsilons} increase, thereby increasing the computational complexity of $\dhat(\Ical, \Jcal)$ as well.
\end{remark}

\section{Lipschitz stability and differentiability of the loss function}\label{sec:vectorization}
The goal of this section is to establish Lipchitz continuity/stability and (almost everywhere) differentiability of our loss function (cf. \Cref{eq:practical_loss}). From Figure \ref{fig:embedding}, recall that we denote the embedding of a $(1,1)$-or $(2,1)$-interval $I$ of $\R^2$ into $\mathbb{R}^6$  by $\mathbf{v}_I$.

\begin{remark}[Independence of variables]\label{rmk:convex_vectorization_for_larger_pq}
Although there are multiple ways to embed $(p, q)$-intervals of $\R^2$ into Euclidean space, for any $p,q\in \N$, 
some vectorization methods are more efficient than others in our context.
For instance, suppose that we represent the $(2,1)$-interval $I$ in Figure \ref{fig:embedding} by simply concatenating the coordinates of its minimal and maximal points. Namely, we represent $I$ with the vector  $(x_1,y_1,x_2,y_2,X,Y)\in \R^6$, where $x_1=x-b$, $y_1=y$, $x_2=x$, $y_2=y-c$, $X=x+d$, $Y=y+a$.
A difficulty with this embedding in the context of gradient descent is that the variables are \emph{not} independent: one must have $(x_1,y_1)\leq (X,Y)$, $(x_2,y_2)\leq (X,Y)$,
and $(x_2,y_2)\not\leq (x_1,y_1)$. Ensuring such relations between variables is not trivial in non-convex optimization (without using refined techniques such as projected gradient descent).
On the other hand, using our proposed embedding is more practical, as all variables are independent, and our only requirements are that $a,b,c,d$ must be nonnegative, which can easily be
imposed using, e.g., exponential or ReLU functions.
\end{remark}
Let $\Ical := \{I_r\}^n_{r=1}$ and $\Jcal := \{J_s\}^m_{s=1}$ consist solely of $(1,1)$- or $(2,1)$-intervals of $\R^2$. 
\begin{lemma} \label{epsilon_IJ}
    For $I_r \in \Ical$ and $J_s \in \Jcal$, consider $\mathbf{v}_{I_r}=(x_1, y_1, a, b, c, d)$ and $\mathbf{v}_{J_s}=(x_2, y_2, e, f, g, h)$ defined as in Figure \ref{fig:embedding}.
    If $||\mathbf{v}_{I_r} - \mathbf{v}_{J_s}||_\infty \leq \epsilon$, then $\epsilon_{rs} \leq 2 \epsilon$ where $\epsilon_{rs}$ is described as in \Cref{thm_distance_btw_collections_intvs}~\ref{item:between_2,1_intervals}. (See \Cref{A:epsilon_IJ} for the proof.)
\end{lemma}
We define $\mathbb{v}_{\Ical}$ as the concatenation of $\mathbf{v}_{I_r}$ for all $I_r \in \mathcal{I}$, i.e., $\mathbb{v}_{\Ical} := (\mathbf{v}_{I_1}|\dots|\mathbf{v}_{I_n}) \in \R^{6n}$.
Let $\mathfrak{I}$ be the collection of all ordered $n$-sets $\Ical$ of $(1, 1)$-or $(2, 1)$-intervals of $\R^2$. Then, define the function $V: \mathfrak{I} \rightarrow \R^{6n}$ by $\Ical\mapsto \mathbb{v}_{\Ical}$.
\begin{proposition}\label{prop:convexity}
    The image of $\mathfrak{I}$ via the function $V$ is a convex subset of $\R^{6n}$.
\end{proposition}
\begin{proof}
    Let $S$ be the collection of all $(1,1)$- and all $(2, 1)$-intervals of $\R^2$. Consider the map from $S$ to $\R\times \R \times \R_{> 0} \times \R_{\geq 0}\times \R_{\geq 0} \times \R_{> 0}$ depicted in Figure \ref{fig:embedding}. Clearly, this map is a surjection.
    It follows that the image of $S$ via the map $I \mapsto (x,y,a,b,c,d)$ depicted in Figure \ref{fig:embedding} is a convex subset of $\R^6$. Now, observe that the image of $\mathfrak{I}$ via the map $V$ is equal to the $n$-fold Cartesian product of the image of $S$, which is a subset of $\R^{6n}$. The fact that any Cartesian product of convex sets is convex implies our claim.
\end{proof}

\begin{theorem}[Lipschitz stability of loss function] \label{thm:de-vectorization_stability} 
    Let $\Kcal$ be any finite set of $(1,1)$- or $(2,1)$-intervals, and let $\Jcal_1, \Jcal_2$ be any two $n$-sets of $(1,1)$- or $(2,1)$-intervals for some $n\in \N^\ast$.
    Then, we have $|\dhat(\Kcal,\Jcal_1)-\dhat(\Kcal,\Jcal_{2})| \leq 2 \cdot \min_{\pi} ||\mathbb{v}_{\pi(\Jcal_1)} - \mathbb{v}_{\Jcal_2}||_\infty,$ where the minimum is taken over all permutations on $\Jcal_1$.
\end{theorem}
\begin{proof}
     By the triangle inequality, we have
     $|\dhat(\Kcal,\Jcal_1)-\dhat(\Kcal,\Jcal_{2})|\leq \dhat(\Jcal_1,\Jcal_2)$ and thus it suffices to show that $\dhat(\Jcal_1,\Jcal_2) \leq  2 \cdot \min_{\pi} ||\mathbb{v}_{\pi(\Jcal_1)} - \mathbb{v}_{\Jcal_2}||_\infty$. 
     Let $\pi_0$ be a permutation on $\Jcal_1$ that attains the minimum. Let $\Ical:=\pi_0(\Jcal_1)$ and $\Jcal:=\Jcal_2$.  By \Cref{epsilon_IJ}, $\epsilon_{rr}$ (resp. $\epsilon_{ss}$) $\leq 2\epsilon$ for all $r, s \in \{1,\dots,n\}$ and hence $\min_s \epsilon_{rs}$ (resp. $\min_r \epsilon_{rs}$) $\leq 2\epsilon$ for each $r$ (resp. $s$). Thus, we have $\max_r (\min_s \epsilon_{rs}) \leq 2\epsilon \mbox{ and } \max_s(\min_r \epsilon_{rs}) \leq 2\epsilon.$ By \Cref{thm_distance_btw_collections_intvs}~\ref{item:between_p,q_intervals}, we are done.    
\end{proof}

\begin{remark} \label{rmk:vectorization_unstable}
    The opposite direction of \Cref{thm:de-vectorization_stability} does not hold, i.e. there does not exist $c > 0$ such that $|\dhat(\Kcal,\Jcal_1)-\dhat(\Kcal,\Jcal_{2})| \geq c\cdot\min_{\pi}||\mathbb{v}_{\pi(\Jcal_1)} - \mathbb{v}_{\Jcal_2}||_\infty.$
    However, the non-existence of such $c > 0$ is \emph{not} an issue in our work as what we require is the Lipschitz stability and almost everywhere differentiability of the loss function (which we establish in the next proposition), in order to prevent erratic and oscillating gradient descent iterations.
\end{remark}

\begin{proposition} \label{prop:differentiable} 
     The loss function given in Equation (\ref{eq:practical_loss}) is differentiable almost everywhere.
\end{proposition}
\begin{proof} 
    The three maps $\R^2\rightarrow \R$ given by $(x,y)\mapsto \delta(x, y)|y-x|$ (cf. Equation (\ref{eq:delta})), $(x,y)\mapsto \max\{x,y\}$, and $(x,y)\mapsto \min\{x,y\}$ are finitely segmented piecewise linear.
    Therefore, the functions given in Equations (\ref{eq_(2, 1)_a})--(\ref{eq_(2, 1)_b'}) are all finitely segmented piecewise linear on Euclidean spaces, and thus so are $\epsilon_{rs}$ given in \Cref{thm_distance_btw_collections_intvs}~\ref{item:between_2,1_intervals}. Hence, trivially, these functions are differentiable almost everywhere. Now, \Cref{eros=intleav} \ref{item:eros-intleav} implies our claim.
\end{proof}

\section{Numerical Experiments on Sparsification} \label{sec:expes}
\input{experiments_clean.tex}

\section{Conclusion}\label{sec:conclusion}
Our sparsification method demonstrates the practicality of approximating full GPDs with sparse ones: while significantly reducing computational costs,
our appropriate loss function also ensures that their discriminative power is not too compromised.
We thus believe that our work paves the way for efficiently deploying multi-parameter topological data analysis to large-scale applications, that are currently out of reach 
for most multi-parameter topological invariants from the literature. In what follows, we outline several potential future directions. 

\subparagraph*{Optimization for finer GPDs.} 
In our experiments, we optimized the GPDs over intervals with at most two minimal points and exactly one maximal point. 
Allowing more complex intervals, as well as dataset-dependent terms in the loss function, could improve performance, but finding suitable embeddings of such intervals into Euclidean space remains a challenge 
(cf. Remark \ref{rmk:convex_vectorization_for_larger_pq}) and would increase the computational cost (cf. Remark \ref{rmk:computational_difficulty_when_pq_gets_larger}).
It would also be interesting to quantify the discriminating power of GPDs with measures that are more direct than the score of a machine learning classifier (or at least to further study the dependencies between GPD sparsification and classifier scores), and to investigate on heuristics (based on drops in the loss values) for deciding whether sparse GPDs are sufficiently good so that optimization can be stopped. 

\subparagraph*{Experimental validation for other GRI-based descriptors.} While we only focused on sparsifying GPDs in this work, it would be interesting
to measure the extent to which our interval domain optimization adapts to other descriptors based on the GRI from the TDA literature; 
of particular interest are the GPDs/signed barcodes coming from \emph{rank exact decompositions}, which have recently proved to be stable~\cite{Botnan2024} (see also Section 8.2 in~\cite{botnan2021signed}
which discusses the influence of the choice of the interval domains on the resulting invariants), as well as GRIL, which focuses on specific intervals called \emph{worms}~\cite{xin2023gril}.

\bibliography{ref.bib}


\appendix

\section{Missing details}

The following remark demonstrates that the bottleneck and Wasserstein distances  \cite{botnan2018algebraic,Botnan2024, loiseaux2024stable} 
are not appropriate candidates for $\dd$ given in  \Cref{eq:ideal_loss}.

\begin{remark}[Instability of the bottleneck distance w.r.t perturbations of the domains of the GPDs]\label{rmk:dB_and_dW_discontinuity} For $n\in \N^\ast$, let $M_n:\R^2\rightarrow \vect$ be the interval module $k_{[0,n]^2}$. For $\epsilon\in \R$, consider the singleton subset $\Ical_n+\epsilon:=\{[\epsilon,n+\epsilon]^2\}$ of $\Int(\R^2)$. Then, it is not difficult to see that  both $\dgm_M^{\Ical_n+\epsilon}$ and $\rk_M^{\Ical_n+\epsilon}$ are defined on the singleton set $\{[\epsilon,n+\epsilon]^2\}$ and \[\dgm_M^{\Ical_n+\epsilon}([\epsilon,n+\epsilon]^2)=\rk_M^{\Ical_n+\epsilon}([\epsilon,n+\epsilon]^2)=\begin{cases}1,&\mbox{if $\epsilon=0$}\\0,&\mbox{otherwise.}\end{cases}
\]
By Remark \ref{rem:monotinocity} \ref{item:dgm-generalizes-barc}, we identify $\dgm_M^{\Ical_n+\epsilon}$ with the set either $\{[0,n]^2\}$ (when $\epsilon=0$) or $\emptyset$ (when $\epsilon\neq 0$) 
(this identification is standard; e.g. \cite{botnan2021signed,kim2021generalized,Scoccola2024}).
Therefore, for any $\epsilon>0$, we have $\db\left(\dgm_M^{\Ical_n+0},\dgm_M^{\Ical_n+\epsilon}\right)=\db(\{[0,n]^2\},\emptyset)=n$
and $\dw{p}\left(\dgm_M^{\Ical_n+0},\dgm_M^{\Ical_n+\epsilon}\right)=\dw{p}(\{[0,n]^2\},\emptyset)=2^{\frac 1p}n$, which can be arbitrarily larger than $\epsilon$, as $n$ increases. 
\end{remark}

\subsection{Proof of \Cref{prop:pseudo}} \label{A:pseudo}
\begin{proof}
    We only prove that $\hatdero$ satisfies the triangle inequality. Let $M, N, O:\R^d\rightarrow \vect$ and let $\Ical, \Jcal, \mathcal{K} \subset \Int(\R^d)$. We show that $\hatdero(\rk^\Ical_M, \rk^\mathcal{K}_O) \leq \hatdero(\rk^\Ical_M, \rk^\Jcal_N) + \hatdero(\rk^\Jcal_N, \rk^\mathcal{K}_O).$
    Assume the following: There exists an $\epsilon_1$-correspondence $\mathcal{R}_1 \subset \mathcal{I} \times \mathcal{J}$ such that 
    \begin{equation} \label{item_eps1corr} 
        \forall (I,J)  \in \mathcal{R}_1,\ \forall\delta \in \R_{\geq 0},\ \ \ \rk_N(J^{\epsilon_1+\delta}) \leq \rk_M(I^\delta) \mbox{ and } \rk_M(I^{\epsilon_1+\delta}) \leq \rk_N(J^\delta).
    \end{equation}
    There exists an $\epsilon_2$-correspondence $\mathcal{R}_2 \subset \Jcal \times \mathcal{K}$ such that 
    \begin{equation}\label{item_eps2corr} 
        \forall (J,K)  \in \mathcal{R}_2,\ \forall\delta \in \R_{\geq 0},\ \ \   \rk_O(K^{\epsilon_2 + \delta}) \leq \rk_N(J^\delta) \mbox{ and } \rk_N(J^{\epsilon_2+\delta}) \leq \rk_O(K^\delta).
    \end{equation}
    Now, consider the following correspondence between $\Ical$ and $\mathcal{K}$:
    \begin{equation*}
        \Rcal_3 := \{(I, K)\in \Ical\times\mathcal{K} : \mbox{ there exists } J \in \Jcal \mbox{ such that } (I,J) \in \Rcal_1 \mbox{ and } (J,K) \in \Rcal_2 \}.
    \end{equation*}
    First, we show that $\Rcal_3$ is an $(\epsilon_1 + \epsilon_2)$-correspondence. Let $(I, K) \in \Rcal_3$. Then, there exists $J \in \Jcal$ such that  $(I,J) \in \Rcal_1 \mbox{ and } (J,K) \in \Rcal_2$. Since 
    $J \subset I^{\epsilon_1}$, we have 
    $J^{\epsilon_2} \subset (I^{\epsilon_1})^{\epsilon_2}= I^{\epsilon_1 + \epsilon_2}.$
    Also, since $K \subset J^{\epsilon_2}$, we have $K \subset I^{\epsilon_1 + \epsilon_2}.$
    Similarly, we can prove $I \subset K^{\epsilon_1 + \epsilon_2}$. This proves that $\Rcal_3$ is an $(\epsilon_1 + \epsilon_2)$-correspondence.  
    Second, by the conditions given in \Cref{item_eps1corr,item_eps2corr}, for all $\delta \in \R_{\geq 0}$, we have:
    \begin{equation*}
        \rk_O(K^{\epsilon_1 + \epsilon_2 + \delta}) \leq \rk_N(J^{\epsilon_1 + \delta}) \leq \rk_M(I^\delta) \mbox{ and } \rk_M(I^{\epsilon_1 + \epsilon_2 + \delta}) \leq \rk_N(J^{\epsilon_2 + \delta}) \leq \rk_O(K^\delta),
    \end{equation*} which completes the proof.
\end{proof}


\subsection{Proof of \Cref{eros=intleav}} \label{A:eros=intleav}
\begin{proof}
    \ref{item:dEhat_is_less_than_dhat}: This directly follows from the definitions of $\dhat$ and $\hatdero$.\\
    \ref{item:hat_generalizes_dE}: 
    Assume that there exists $\epsilon\in \R_{\geq 0}$ such that for all $I\in \Ical$, $\rk_N(I^{\epsilon}) \leq \rk_M(I) \mbox{ and } \rk_M(I^{\epsilon}) \leq \rk_N(I)$. Then, $\Rcal:=\{(I,I):I\in \Ical\}$ is clearly a $0$-correspondence (and thus also an $\epsilon$-correspondence), and satisfies
    \[ \mbox{for all }(I, I) \in \mathcal{R} \mbox{ and for all } \delta \in \R_{\geq 0},\; \rk_N(I^{\epsilon+\delta}) \leq \rk_M(I^\delta) \mbox{ and } \rk_M(I^{\epsilon+\delta}) \leq \rk_N(I^\delta),\]
    which completes the proof.\\
    \ref{item:eros-intleav}: 
    Let $\mathcal{R}$ be an $\epsilon$-correspondence between $\Ical$ and $\Jcal$. Since $J \subset I^{\epsilon}$, we have that for all $\delta \in \R_{\geq 0}$, $J^\delta \subset (I^\epsilon)^\delta = I^{\epsilon + \delta}$. Similarly, we have that $I^\delta \subset J^{\epsilon + \delta}$. By monotonicity of $\rk_M$ (Remark \ref{rem:monotinocity} \ref{item:monotonicity}), it follows that for all $(I, J) \in \mathcal{R}$ and for all $\delta \in \R_{\geq 0}$, $\rk_M(J^{\epsilon+\delta}) \leq \rk_M(I^\delta)$ and $\rk_M(I^{\epsilon+\delta}) \leq \rk_M(J^\delta)$. This directly implies the claim.
\end{proof}


\subsection{Proof of \Cref{lem_epsrs}} \label{A:lem_epsrs}
\begin{proof}
    The claim directly follows by proving the following two equivalences:
    \begin{align*}
        J_s \subset I^{\epsilon}_r \iff \epsilon \geq \max\left(\max_k a_k, \max_l b_l\right),\\   
        I_r \subset J^{\epsilon}_s \iff \epsilon \geq \max\left(\max_i a'_i, \max_j b'_j\right).
    \end{align*}
    We prove the first equivalence. The second equivalence is similarly proved.

    We have that $J_s \subset I^{\epsilon}_r$ iff for all $(x^s_k, y^s_k) \in \min(J_s),$ there exists $(x^r_i, y^r_i) \in \min(I_r)$ such that
    \begin{equation} \label{eq_min(I)-eps_leq_min(J)}
     (x^r_i,y^r_i) - \epsilon \leq (x^s_k,y^s_k),
    \end{equation}
    and also for all $(X^s_l, Y^s_l) \in \max(J_s),$ there exists $(X^r_j, Y^r_j) \in \max(I_r)$ such that
    \begin{equation} \label{eq_max(J)_leq_max(I)+eps}
        (X^s_l, Y^s_l) \leq (X^r_j, Y^r_j) + \epsilon.
    \end{equation}
    Since Inequality (\ref{eq_min(I)-eps_leq_min(J)}) is equivalent to \[
    \epsilon \geq \max \Bigl( (1-\delta(x^r_i, x^s_k))|x^s_k - x^r_i|, (1-\delta(y^r_i, y^s_k))|y^s_k - y^r_i| \Bigl),
    \]
    for any $(x^s_k, y^s_k) \in \min(J_s)$, there exists $(x^r_i, y^r_i) \in \min(I_r)$ satisfying  (\ref{eq_min(I)-eps_leq_min(J)}) if and only if $\epsilon \geq$ (\ref{eq_a}).
    Similarly, for each $(X^s_l, Y^s_l) \in \max(J_s)$, there exists $(X^r_j, Y^r_j) \in \max(I_r)$ satisfying (\ref{eq_max(J)_leq_max(I)+eps}) if and only if $\epsilon \geq$ (\ref{eq_b}). To sum up, we have:
    \[J_s \subset I^{\epsilon}_r \Leftrightarrow \epsilon \geq \max(\max_k a_k, \max_l b_l).\]
\end{proof}
\subsection{Proof of \Cref{thm_distance_btw_collections_intvs}} \label{A:pf_(2, 1)}
\begin{proof}
    \ref{item:between_p,q_intervals}:  To begin, note that there exists $\epsilon$-correspondence $\mathcal{R} \subset \mathcal{I} \times \mathcal{J}$ if and only if for all $I_r \in \Ical$ (resp. $J_s \in \Jcal$), there exists $J_s \in \Jcal$ (resp. $I_r \in \Ical$) such that 
    \begin{equation} \label{eq_epsinclusion}
        J_s \subset I^{\epsilon}_r, \quad I_r \subset J^{\epsilon}_s
    \end{equation}
    By \Cref{lem_epsrs}, for each $I_r \in \Ical$, there exists $J_s \in \Jcal$ satisfying (\ref{eq_epsinclusion}) if and only if
    \begin{equation*}
        \epsilon \geq \min_s \epsilon_{rs} \mbox{ for all } r.
    \end{equation*}
    Likewise, for each $J_s \in \Jcal$, there exists $I_r \in \Ical$ satisfying (\ref{eq_epsinclusion}) if and only if
    \begin{equation*}
        \epsilon \geq \min_r \epsilon_{rs} \mbox{ for all } s.
    \end{equation*}
    Therefore, for all $I_r \in \Ical$ (resp. $J_s \in \Jcal$) there exists $J_s \in \Jcal$ (resp. $I_r \in \Ical$) if and only if
    \begin{equation} \label{eq_eps} 
        \epsilon \geq \max \Bigl(\max_r(\min_s \epsilon_{rs}), \max_s(\min_r \epsilon_{rs}) \Bigl).
    \end{equation}
    Since there exists $\epsilon$-correspondence $\mathcal{R} \subset \Ical \times \Jcal$ for $\epsilon$ satisfying (\ref{eq_eps}), by the definition of $\dhat$, we have our desired result.\\
    \ref{item:between_2,1_intervals}: It suffices to show that $\max_k a_k$, $\max_l b_l$, $\max_i a'_i$, $\max_j b'_j$ are equal to (\ref{eq_(2, 1)_a}), (\ref{eq_(2, 1)_b}), (\ref{eq_(2, 1)_a'}), and (\ref{eq_(2, 1)_b'}) respectively.
    Indeed, since both $i, j$ range from $1$ to $2$, and $k=l=1$, it follows from (\ref{eq_a}) that $\max_k a_k$ in \Cref{eq:epsilons} is equal to 
    \begin{equation*}
    \begin{split}
        \max \Biggl( \min \biggl(& \max \Bigl( (1 - \delta(x_1 - b \leq x_2 - f))|x_2 - f - x_1 + b|, (1 - \delta(y_1 \leq y_2))|y_2 - y_1| \Bigl),\\
        & \max \Bigl((1 - \delta(x_1 \leq x_2 - f))|x_2 - f - x_1|, (1 - \delta(y_1 - c\leq y_2))|y_2 - y_1 + c| \Bigl) \biggl),\\
        \min \biggl(& \max \Bigl((1 - \delta(x_1 - b \leq x_2))|x_2 - x_1 + b|, (1 - \delta(y_1 \leq y_2 - g))|y_2 - g - y_1| \Bigl),\\
        & \max \Bigl((1 - \delta(x_1 \leq x_2))|x_2 - x_1|, (1 - \delta(y_1 - c \leq y_2 - g))|y_2 - g - y_1 + c| \Bigl) \biggl) \Biggl).
    \end{split}
    \end{equation*}
    Since $1 - \delta(x, y) = \delta(y, x)$ (cf. Equation (\ref{eq:delta})), by the definitions of $F$ and $H$ (see right above \Cref{thm_distance_btw_collections_intvs}), we have that $\max_k a_k$ in \Cref{eq:epsilons} equals (\ref{eq_(2, 1)_a}). 
    Similarly, the other terms of the maxima on the right-hand side of Equation (\ref{eq:epsilons}) is also equal to a term of (\ref{eq_(2, 1)_b})-(\ref{eq_(2, 1)_b'}).
\end{proof}


\subsection{Proof of \Cref{epsilon_IJ}} \label{A:epsilon_IJ}
\begin{proof}
    It suffices to show that all terms (\ref{eq_(2, 1)_a})-(\ref{eq_(2, 1)_b'}) are less than or equal to $2\epsilon$.  
    Since $||\mathbf{v}_{I_r} - \mathbf{v}_{J_s}||_\infty \leq \epsilon$ implies that \[|x_1 - x_2|, |y_1 - y_2|, |a - e|, |b - f|, |c - g|, |d - h| \leq \epsilon,\] 
    we have that $|x_1 - b - x_2 + f| \leq |x_1 - x_2| + |f - b| \leq 2\epsilon$. Similarly, $|y_1 - c - y_2 + g|, |x_2 - d  - x_1 + h|, |y_2 -a - y_1 + e| \leq 2\epsilon$, which implies that (\ref{eq_(2, 1)_a}) $\leq 2 \epsilon$. Indeed, 
    since $|x_1 - b - x_2 + f| \leq 2\epsilon$ and thus $\delta(x_2 - f \leq x_1 - b)|x_1 - b - x_2 + f| \leq 2\epsilon$ and in turn \[
    F(x_2-f, x_1-b, y_1, y_2) = \max(\delta(x_2 - f \leq x_1 - b)|x_1 - b - x_2 + f|, \delta(y_1, y_2)|y_2 - y_1|) \leq 2\epsilon. 
    \] 
    Similarly, we have that \[
    F(x_1, x_2, y_2-g,y_1-c) = \max(\delta(x_1, x_2)|x_2 - x_1|, \delta(y_2 - g, y_1 - c)|y_1 - c - y_2 + g|) \leq 2\epsilon.
    \] Thus, (\ref{eq_(2, 1)_a}) is less than or equal to  $2\epsilon$. Via similar arguments, one can show that the other terms (\ref{eq_(2, 1)_b})-(\ref{eq_(2, 1)_b'}) are also less than or equal to $2\epsilon$, which completes the proof.
\end{proof}

\newpage

\section{Additional tables}

\begin{table}[h] 
\centering 
\begin{tabular}{ll}
\toprule
Dataset & Acronym  \\ 
\midrule
\texttt{Coffee} & C
\\ 
\texttt{DistalPhalanxOutlineAgeGroup} & DPA
\\ 
\texttt{DistalPhalanxOutlineCorrect} & DPC
\\ 
\texttt{DistalPhalanxTW} & DPT
\\ 
\texttt{ProximalPhalanxOutlineAgeGroup} & PPA
\\ 
\texttt{ProximalPhalanxOutlineCorrect} & PPC
\\
\texttt{ProximalPhalanxTW} & PPT
\\ 
\texttt{ItalyPowerDemand} & IPD
\\ 
\texttt{ECG200} & ECG
\\ 
\texttt{MedicalImages} & MI
\\ 
\texttt{Plane} & P
\\ 
\texttt{SwedishLeaf} & SL
\\ 
\texttt{GunPoint} & GP
\\ 
\texttt{GunPointAgeSpan} & GPA
\\ 
\texttt{GunPointMaleVersusFemale} & GPM
\\ 
\texttt{GunPointOldVersusYoung} & GPO
\\
\texttt{PowerCons} & PC
\\ 
\texttt{SyntheticControl} & SC
\\ 
\bottomrule
\end{tabular}
\caption{\label{tab:acronyms} UCR dataset acronyms.}
\end{table}

\end{document}

%% file: experiments_clean.tex
In this section, we make use of the results proved above
to provide a method for {\em sparsifying} GPDs.
Indeed, computing a single GPD $\dgm^\Ical_M$ on a persistence module $M$ coming from a simplicial complex $S$ 
requires computing zigzag persistence modules on all intervals of $\Ical$, as described in~\cite{dey2022computing},
and then computing the corresponding M\"obius inversion, yielding a time complexity of $O(n^3)$, where $n$ is the number of intervals in $\Ical$, due to the need to compute the M\"obius function value for all possible pairs of intervals, each requiring $O(n)$ operations to iterate over each interval. Thus, the complexity of computing a single GPD is $O(n\times N_s^{2.376} + n^3)$, where 
$N_s$ is the number of simplices \cite{milosavljevic2011zigzag}. 
Hence, computing all GPDs from a dataset of persistence modules becomes intractable when $n$ is large.

Therefore, as described in the introduction, our goal in 
this section is to design a {\em sparse} subset of $(2,1)$-intervals $\Jcal^*$ of size $m\ll n$ that minimizes 
the loss function in~\Cref{eq:practical_loss}
with gradient descent by treating every interval in $\Jcal$ as a parameter to optimize. 

\paragraph{Time-series datasets.} 
The datasets we consider are taken from the UCR repository~\cite{Dau2018}, and correspond to classification tasks that have already been studied 
with persistent homology before~\cite[Section 4.2]{loiseaux2024stable}. More precisely, instances in these datasets take the form of labelled time series, that we pre-process with time-delay embedding
in order to turn them into point clouds. Specifically, each labelled time series $T=\{f(t_1),\dots,f(t_n)\}$ of length $n$ is transformed into a point cloud $X_T\subset\mathbb{R}^3$ of cardinality $n-2$ with 
$$X_T:=\{(f(t_1),f(t_2),f(t_3))^T,\dots,(f(t_{n-2}),f(t_{n-1}),f(t_{n}))^T\}.$$

Then, we compute both the Vietoris-Rips filtration and the sublevel set filtration induced by a Gaussian kernel density estimator 
(with bandwidth $\sigma=0.1 \cdot {\rm diam}(X_T)$, where ${\rm diam}(X_T):=\max_{x,y\in X_T}\|x-y\|_2$ is the diameter of the point cloud)
using the \texttt{PointCloud2FilteredComplex} function of the \texttt{multipers}
library~\cite{Loiseaux2024}.
\footnote{See the tutorial available at \url{https://davidlapous.github.io/multipers/notebooks/time_series_classification.html} for a 
detailed description of the procedure.} 
Both filtrations are then normalized so that their ranges become equal to the unit interval $[0,1]$.

\paragraph{Loss function.} 
In order to compute GPDs out of these 2-parameter filtrations and modules, one needs a subset of intervals.
As explained in the introduction, we then minimize 
$$\Lcal_{\hatdero,m}: \mathbb{v}_\Jcal \mapsto \dhat(\Ical, \Jcal),$$

where the full domain $\Ical$ (resp. the sparse domain $\Jcal$) is comprised of $n=1,600$ (resp. $m=400$) $(2,1)$-intervals obtained from a grid in $\mathbb{R}^6$ 
computed by evenly sampling $10$ (resp. $5$) values for $x$ and $y$ and $2$ values for $a,b,c,d$ within their corresponding filtration ranges.
Note that while the $(2,1)$-intervals in $\mathcal J$ are treated as parameters to optimize, the $(2,1)$-intervals in $\mathcal I$ are fixed throughout the optimization process.
Moreover, the formula that we provided in \Cref{thm_distance_btw_collections_intvs}~\ref{item:between_2,1_intervals} can be readily implemented in any library that uses auto-differentiation, such as $\texttt{pytorch}$. 
In particular, we run stochastic gradient descent with momentum $0.9$ on $\mathcal L$ for $750$ epochs using learning rate $\eta=0.001$ with exponential decay of factor $0.99$ to achieve convergence,
and obtain our sparse subset $\mathcal J^*$. See Figure~\ref{fig:loss_optim} for a visualization of the loss decrease. Note how the Lipschitz stability proved in \Cref{thm:de-vectorization_stability}
translates into a smooth decrease with small oscillations. 

\begin{figure}[h]
  \begin{center}
    \includegraphics[width=.5\textwidth]{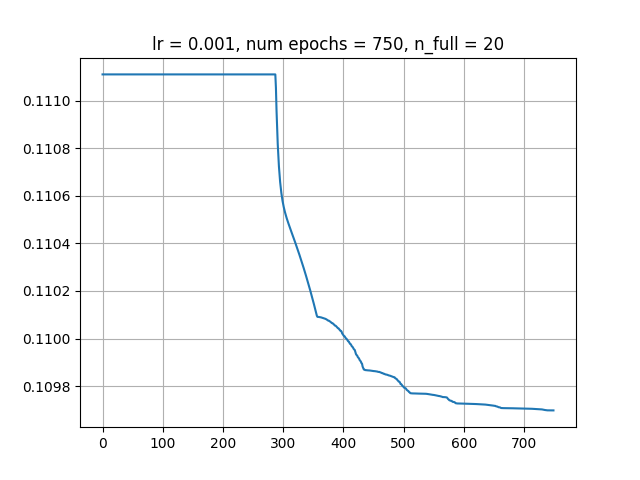}
  \end{center}
  \caption{\label{fig:loss_optim} Loss decrease across gradient descent iterations. One can see that the loss value stays on a plateau for the first $\sim$300 iterations; this is due to the fact that during these first iterations,
the parameters in $\Jcal$ that are updated with gradient descent are not yet the ones achieving the maxima and minima in the closed-form formula provided in \Cref{thm_distance_btw_collections_intvs}~\ref{item:between_2,1_intervals}.}
\end{figure}

\subparagraph*{Accuracy scores.} 
In order to quantify the running time improvements as well as the information loss (if any) when switching from the full domain $\Ical$
to the optimized sparse one $\Jcal^*$, we computed the full and sparse GPDs in homology dimensions $0$ and $1$ with the zigzag persistence diagram 
implementation in \texttt{dionysus},\footnote{\url{https://mrzv.org/software/dionysus2/}} and then 
we trained random forest classifiers on these GPDs to predict the time series labels.  
In order to achieve this, we first turned both the full and optimized GPDs into Euclidean vectors by first binning every GPD (seen as a point cloud in $\mathbb{R}^6$) with a fixed $6$-dimensional histogram,
and then convolving this histogram with a $6$-dimensional Gaussian kernel in order to smooth its values, with a procedure similar to the one described in~\cite[Section 3.2.1]{loiseaux2024stable}. 
Both the histogram bins and the kernel bandwidths were found with $3$-fold cross-validation on the training set (see the provided code for 
hyperparameter values). Then, random forest classifiers were trained on these smoothed histograms to predict labels; 
in Table~\ref{table_scores} we report the accuracy scores of these classifiers for the initial (before optimization) sparse domain $\Jcal_{\rm init}$, the optimized sparse domain $\Jcal^*$, and 
the full domain $\Ical$. Moreover, in Table~\ref{table_times}, we report the running time needed to compute all GPDs using either $\Jcal_{\rm init}$, $\Jcal^*$, or $\Ical$, as well as the improvement
when switching from $\Jcal^*$ to $\Ical$. Note that our goal is not to improve on the state-of-the-art for time series classification, but rather to assess
whether optimizing the loss $\Lcal_{\hatdero, m}$ based on our upper bound is indeed beneficial for improving topology-based models. See Figure~\ref{fig:pipeline} for a schematic overview of our full pipeline.

\subparagraph*{Discussion on results.} 
As one can see from Table~\ref{table_scores}, there is either a clear improvement or a comparable performance in accuracy scores after optimizing $\Jcal$. Indeed, by minimizing $\Lcal_{\hatdero,m}$, one forces the 
sparse domain $\Jcal$ to be as close as possible to the full domain $\Ical$, and thus to retain as much topological information as possible. 
Hence, either the full GPDs are more efficient than the initial sparse GPDs, in which case the optimized GPDs perform much better than the initial sparse ones,
or the full GPDs are less efficient than the initial sparse GPDs, 
\footnote{Recall that the accuracy score is only an \emph{indirect} measure: while the full GPDs are \emph{always} richer in topological information, their scores might still be
lower due to, e.g., many empty or redundant intervals.} 
in which case the optimized GPDs have comparable or slightly worse performances than the initial sparse ones 
thanks to the small sizes of their domains
(except for the PC and IPD datasets, for which the optimized GPDs still perform interestingly better).
In all cases, we emphasize that optimized GPDs provide the best solution: 
\emph{they maintain scores at levels that are either comparable or better than the best solution between the initial sparse and full GPDs,
while avoiding to force users to choose interval domains} (as their domains are obtained automatically with gradient descent).

As for running times, we observe a slight increase from the running times associated to the initial sparse GPDs (except for the IPD dataset), which 
is due to the use of optimized intervals that are richer in topological information than the initial ones, and a strong improvement
over the running times associated to the full GPDs, with ratios ranging between $5$ and $15$ times faster.
\emph{Our optimized GPDs thus achieve the best of both worlds: they are strinkingly fast to compute while keeping high accuracy scores.}
Our code was run on a \texttt{2x Xeon SP Gold 5115 @ 2.40GHz}, and is fully available at \codelink.

\begin{table}[h]
\input{scores.tex}
\caption{\label{table_scores} Accuracy scores ($\%$) of random forest classifiers trained on several UCR datasets. Underline indicates best score between the initial and optimized sparse domains, while bold 
font indicates best score overall. See Table~\ref{tab:acronyms} for the dataset full names.}
\end{table}

\begin{table}[h]
\input{times.tex}
\caption{\label{table_times} Running times (seconds) needed for computing all GPDs on several UCR datasets. Underline indicates best running time between the initial and optimized sparse domains, while bold 
font indicates best running time overall. See Table~\ref{tab:acronyms} for the dataset full names.}
\end{table}

%% file: scores.tex
\begin{tabular}{l|ccccccccc}
\hline 
  & C & DPA & DPC & DPT & PPA & PPC & PPT & ECG & IPD \\ 
\hline 
Init. & 0.750 & \underline{\textbf{0.705}} & \underline{\textbf{0.746}} & \underline{\textbf{0.561}} & \underline{\textbf{0.790}} & 0.718 & 0.693 & \underline{\textbf{0.790}} & 0.677 \\
Optim. & \underline{0.786} & 0.691 & 0.721 & \underline{\textbf{0.561}} & 0.785 & \underline{\textbf{0.742}} & \underline{\textbf{0.737}} & \underline{\textbf{0.790}} & \underline{\textbf{0.690}} \\
Full & \textbf{0.857} & 0.669 & 0.743 & 0.554 & 0.780 & 0.729 & 0.707 & 0.740 & 0.651 \\
\hline
\hline
& MI & P & SL & GP & GPA & GPM & GPO & PC & SC \\
\hline
Init. & 0.534 & \underline{\textbf{0.924}} & \underline{\textbf{0.565}} & \underline{\textbf{0.900}} & 0.835 & 0.946 & \underline{0.987} & 0.789 & 0.447 \\ 
Optim. & \underline{\textbf{0.588}} & 0.886 & 0.546 & 0.893 & \underline{\textbf{0.959}} & \underline{\textbf{0.959}} & 0.956 & \underline{\textbf{0.800}} & \underline{0.487} \\ 
Full & 0.545 & 0.838 & 0.531 & 0.847 & 0.864 & 0.946 & \textbf{0.990} & 0.778 & \textbf{0.503} \\ 
\hline
\end{tabular}

%% file: times.tex
\begin{tabular}{l|ccccccccc}
\hline 
  & C & DPA & DPC & DPT & PPA & PPC & PPT & ECG & IPD \\ 
\hline 
Init. & \underline{\textbf{350}} & \underline{\textbf{1076}} & \underline{\textbf{1902}} & \underline{\textbf{1178}} & \underline{\textbf{1182}} & \underline{\textbf{1660}} & \underline{\textbf{1147}} & \underline{\textbf{556}} & \underline{\textbf{1262}} \\
Optim. & 421 & 1172 & 2054 & 1265 & 1257 & 1749 & 1232 & 621 & 1328 \\
Full & 2304 & 11672 & 20163 & 13007 & 12745 & 18427 & 12844 & 5098 & 19860 \\
\hline 
Improv. & 5.47x & 9.95x & 9.82x & 10.28x & 10.13x & 10.53x & 10.42x & 8.20x & 15.74x \\
\hline
\hline
 & MI & P & SL & GP & GPA & GPM & GPO & PC & SC \\
\hline
Init. & \underline{\textbf{2989}} & \underline{\textbf{792}} & \underline{\textbf{3478}} & \underline{\textbf{703}} & \underline{\textbf{1576}} & \underline{\textbf{1658}} & \underline{\textbf{1573}} & \underline{\textbf{1339}} & \underline{\textbf{1240}} \\ 
Optim. & 3399 & 909 & 3919 & 883 & 1955 & 2061 & 1925 & 1550 & 1324 \\ 
Full & 29074 & 6280 & 31285 & 5773 & 12913 & 13523 & 12849 & 10730 & 13089 \\ 
\hline 
Improv. & 8.55x & 6.91x & 7.98x & 6.54x & 6.60x & 6.56x & 6.67x & 6.92x & 9.88x \\ 
\hline
\end{tabular}